\documentclass{amsart}

\usepackage{amsmath,amsthm,amssymb,relsize}
\usepackage[nobysame,alphabetic,initials]{amsrefs}
\usepackage[margin=30mm]{geometry}
\usepackage{mathrsfs,mathtools,makecell}

\usepackage{booktabs, multirow, adjustbox, array}

\usepackage{enumerate}

\usepackage{tikz}
\usetikzlibrary{
  cd,
  calc,
  positioning,
  fit,
  arrows,
  decorations.pathreplacing,
  decorations.markings,
  shapes.geometric,
  backgrounds,
  bending
}
\usepackage{tikzsymbols}
\usepackage{xcolor}

\definecolor{darkblue}{rgb}{0,0,0.6}

\usepackage[capitalize,noabbrev]{cleveref}

\makeatletter
\newtheorem*{rep@theorem}{\rep@title}
\newcommand{\newreptheorem}[2]{%
\newenvironment{rep#1}[1]{%
 \def\rep@title{#2 \ref{##1}}%
 \begin{rep@theorem}}%
 {\end{rep@theorem}}}
\makeatother

\numberwithin{equation}{section}

\newtheorem{proposition}[equation]{Proposition}

\newtheorem{lemma}[equation]{Lemma}

\newtheorem{thmx}{Theorem}

\theoremstyle{definition}

\theoremstyle{remark}
\newtheorem{remark}[equation]{Remark}

\newtheorem*{remark*}{Remark}

\newreptheorem{theorem}{Theorem}
\newreptheorem{lemma}{Lemma}
\newreptheorem{proposition}{Proposition}
\newreptheorem{corollary}{Corollary}
\newreptheorem{question}{Question}
\numberwithin{equation}{section}

\newcommand{\Z}{\mathbb{Z}}

\newcommand{\Id}{\operatorname{Id}}

\newcommand{\ol}{\overline}
\newcommand{\wt}{\widetilde}

\newcommand{\sm}{\setminus}

\newcommand{\Top}{\mathrm{Top}}

\newcommand{\bsm}{\left(\begin{smallmatrix}}
\newcommand{\esm}{\end{smallmatrix}\right)}

\usepackage{letltxmacro}

\LetLtxMacro\Oldfootnote\footnote

\begin{document}
\title[Spanning 3-discs for the trivial 2-link]{Spanning 3-discs in the 4-sphere pushed into the 5-disc}

\author{Mark Powell}
\address{School of  Mathematics and Statistics, University of Glasgow, United Kingdom}
\email{mark.powell@glasgow.ac.uk}

\def\subjclassname{\textup{2020} Mathematics Subject Classification}
\expandafter\let\csname subjclassname@1991\endcsname=\subjclassname

\begin{abstract}
I prove that any two smooth collections of spanning 3-discs for the trivial 2-link in $S^4$ become smoothly isotopic rel.\ boundary after pushing them into $D^5$.
\end{abstract}
\maketitle

\section{Introduction}

An \emph{$m$-component 2-link} is a smooth submanifold of $S^4$ homeomorphic to a disjoint union $\sqcup^m S^2$.
Given an $m$-component 2-link $L$, a smoothly embedded collection of 3-discs $D_m \subseteq S^4$, with $D_m \cong \sqcup^m D^3$ and $\partial D_m = L$, is called a \emph{spanning 3-disc collection} for $L$.
If $L$ admits a spanning 3-disc collection, I say that $L$ is \emph{trivial}. Fix a standard trivial $m$-component 2-link $U_m \subseteq S^4$.

\begin{thmx}\label{thm}
  Let $D_m^0$ and $D_m^1$ be spanning 3-disc collections for the trivial 2-link $U_m$ in $S^4$. Then including $S^4 \subseteq D^5$, $D^0_m$ and $D^1_m$ become smoothly isotopic in $D^5$, relative to $U_m$.
\end{thmx}

\begin{remark}
   In the case that $m=1$, i.e.\ spanning 3-discs for the trivial 2-knot, this was proven by Hartman in \cite{H}. An alternative argument was given in the introduction to Hughes-Kim-Miller~\cite{HKM}.  The $m=1$ case of the proof of \cref{thm} gives another proof that is more direct.  The case of multiple connected components is new. I deduce it from an application of surgery theory.
\end{remark}

Since the surgery methods used in the proof have been available for some time, let me mention how this question arose in the modern context.
Budney--Gabai~\cite{BG} showed that there are spanning 3-discs for $U_1$ that are not isotopic rel.\ boundary in $S^4$, but which become isotopic in $D^5$.
The $m=1$ case of \cref{thm} showed that the Budney--Gabai examples are optimal in the sense that it is not possible to construct examples that remain distinct in $D^5$.

Alison Tatsuoka informs me that work in progress with Seungwon Kim and Gheehyun Nahm will produce spanning 3-disc collections for $U_m$, $m \geq 2$, that are not isotopic rel.\ boundary. Their collections are Brunnian, in the sense that every $(m-1)$-component sub-collection is isotopically standard.
In his PhD thesis, Weizhe Niu~\cite{Niu} independently developed alternative methods, extending the computations from~\cite{BG}, that can likely be applied to prove the same result; this is also work in progress of Niu.
These new collections can be seen directly to be standard in $D^5$; \cref{thm} shows that this must always be the case.

If I allow trivial surface links in $S^4$ with positive genus, there is no analogous result.  Hughes--Kim--Miller~\cite{HKM} showed that there are pairs of genus $g \geq 2$ handlebodies in $S^4$, with the same surface boundary, that are not isotopic rel.\ boundary, and remain non-isotopic rel.\ boundary after they are pushed into $D^5$.  The genus one case is open at the time of writing, but Kim--Nahm--Tatsuoka have announced similar examples in this case too.
Assuming their veracity, it follows that for every trivial surface link in $S^4$ that contains a surface of positive genus, there exist distinct pairs of spanning handlebodies that remain distinct rel.\ boundary in $D^5$.

Thus, Theorem~\ref{thm} is optimal: for the isotopy uniqueness result it espouses to hold, one must allow pushing the spanning 3-discs into the 5-disc, and one cannot increase the topological complexity beyond spherical 2-links.

\begin{remark}
  The proof I give of \cref{thm} is analogous to the proof that slice discs $D^2 \subseteq D^4$ for Alexander polynomial one knots are unique up to topological isotopy rel.\ boundary, which was proven in my work with Conway~\cite{CP}. Since \cref{thm} concerns one dimension up, the result holds in the smooth category. In addition, I can work here with free groups of arbitrary rank, whereas in the 4-dimensional case we needed to work in a setting where the fundamental group was good.
\end{remark}

\subsubsection*{Acknowledgements}

Thanks to Seungwon Kim and Maggie Miller for suggesting this question, and to both Anthony Conway and the anonymous referee for helpful comments on previous drafts.  I am very grateful to Sarah Blackwell and Alison Tatsuoka for informing me about potential applications and comparisons.
I was partially supported by EPSRC grants EP/T028335/1 and  EP/V04821X/1.


\section{Using surgery to compare pushed-in 3-disc exteriors}

All manifolds and embeddings are assumed to be smooth. Fix $m \geq 1$.
Push $D^0_m$ and $D^1_m$ into $D^5$, and by an abuse of notation denote these pushed in copies also by $D^0_m$ and $D^1_m$ respectively.  Write $W_i := D^5 \sm \nu D^i_m$ for $i=0,1$.

\begin{lemma}\label{lemma:DONM}
  For $i=0,1$, there is a degree one normal map of pairs $(f_i,b_i) \colon (W_i,\partial W_i) \to (\natural^m S^1 \times D^4, \#^m S^1 \times S^3)$ such that $f_i$ is a homotopy equivalence that restricts to a diffeomorphism on $\partial W_i$.
\end{lemma}

\begin{proof}
   To begin, I construct a diffeomorphism $\varphi \colon \partial W_0 \xrightarrow{\cong} \partial W_1$, that restricts to the identity between the two copies of $S^4 \sm \nu U_m$ and to diffeomorphisms between the two copies of \[\sqcup^m D^3 \times S^1 \cong \partial W_i \sm (S^4 \sm \ol{\nu} U_m).\]
   Add an exterior collar $S^4 \times I$ to $\partial D^5$, such that $S^4 \times \{0\}$ is attached to $\partial D^5$.  Consider the 5-manifold
   \[Y_i := (S^4 \times I) \cup \ol{\nu} D^i_m \cong (S^4 \times I) \cup (\sqcup^m D^3 \times D^2),\]
    which has boundary $S^4 \sqcup \partial W_i$. Fix an identification $\ol{\nu} D^0_m \cong \ol{\nu} D^1_m$ that sends the 0-section to the 0-section and extends the identity on $\ol{\nu} D^0_m \cap S^4 \to \ol{\nu} D^1_m \cap S^4$.  Extending by the identity on $S^4 \times I$, this yields a diffeomorphism $Y_0 \cong Y_1$ that restricts to the identity on $S^4 \times \{0\}$ and sends $D^0_m$ to $D^1_m$. Restricting to the $\partial W_0$  boundary component yields the desired diffeomorphism $\varphi \colon \partial W_0 \cong \partial W_1$.

      I then choose compatible identifications of both $\partial W_0$ and $\partial W_1$ with $\#^m S^1 \times S^3$, for example by choosing one for $\partial W_1$ and then using $\varphi$ to obtain the identification for $\partial W_0$. (The careful construction of this diffeomorphism $\varphi$, and the fact that the identifications with $\#^m S^1 \times S^3$ are compatible,  will be used in \cref{sec:proof-thm-A}.)

  For each $i=0,1$, extend the diffeomorphism $\partial W_i \cong \#^m S^1 \times S^3$ to a diffeomorphism of collar neighbourhoods $\partial W_i \times [0,1]$ to $(\#^m S^1 \times S^3) \times [0,1]$.  Compose with the projection $(\#^m S^1 \times S^3) \times [0,1] \to \#^m S^1 \times S^3$, followed by a standard Pontryagin-Thom type map $\#^m S^1 \times S^3 \to \vee^m S^1$  to obtain a map $\partial W_i \times  [0,1] \to \vee^m S^1$. Using obstruction theory, extend this to a map $g \colon W_i \to \vee^m S^1$.  Here note that $\pi_j(\vee^m S^1)=0$ for $j >1$, so as long as I define the map to $\vee^m S^1$ correctly on the 1-cells, the rest of the obstruction theory proceeds without hindrance.

Next, the open collar neighbourhood $\partial W_i \times [0,1)$ maps to $(\natural^m S^1 \times D^4) \sm (\vee^m S^1)$ via a diffeomorphism.
Send the rest of $W_i$ to the core $\vee^m S^1$ using $g$. I shall argue that the resulting map $f_i \colon W_i \to \natural^m S^1 \times D^4$ is a homotopy equivalence. Note that $\pi_1(W_i) \cong F_m$, the free group of rank $m$. The map I have defined induces an isomorphism $\pi_1(W_i) \xrightarrow{\cong} \pi_1(\natural^m S^1\times D^4)$.

Push $D^i_m$ into $D^5$ such that the radial function restricted to $D^i_m$ is a Morse function with $m$ critical points, each of which has index zero. It follows that the exterior has a handle decomposition with a single 0-handle and $m$ 1-handles, and is therefore diffeomorphic (without any control on the diffeomorphism on the boundary) to $\natural^m S^1 \times D^4$. In particular $W_i$ is homotopy equivalent to $\vee^m S^1$, and so any map $W_i \to \vee^m S^1$ inducing an isomorphism on  $\pi_1$, such as the map $f_i$ under consideration, is a homotopy equivalence by Whitehead's theorem.

Since $\natural^m S^1 \times D^4$ and $W_i$ have trivial tangent bundles, our homotopy equivalence can be augmented with the necessary bundle data to obtain a normal map.  As $f_i$ restricts to a diffeomorphism $f_i| \colon \partial W_i \to \# S^1 \times S^3$, it is automatically degree one.
\end{proof}

I can now consider $f_i \colon W_i \to \natural^m S^1 \times D^4$ as an element of the
Browder-Novikov-Sullivan-Wall relative surgery structure set~\cite{W}, denoted  $\mathcal{S}(\natural^m S^1 \times D^4, \#^m S^1 \times S^3)$, which sits in the exact sequence (of abelian groups, in this case, since $\natural^m S^1 \times D^4 \cong D^2 \times \natural^m S^1 \times D^2$):
\begin{equation}\label{eqn:ses}
L_6(\Z[F_m]) \to \mathcal{S}(\natural^m S^1 \times D^4, \#^m S^1 \times S^3) \to \mathcal{N}(\natural^m S^1 \times D^4, \#^m S^1 \times S^3) \to L_5(\Z[F_m]).
\end{equation}
Note that since the Whitehead group of $F_m$ is trivial~\cite{St1}, I can ignore $h$ and $s$ decorations, and every $h$-cobordism is an $s$-cobordism. Let $\wt{L}_n(\Z[\pi])$ denote the reduced $L$-theory of $\Z[\pi]$, such that $L_n(\Z[\pi]) \cong L_n(\Z) \oplus \wt{L}_n(\Z[\pi])$.
Recall that $L_n(\Z)=0$ for $n$ odd, and for $k \geq 0$ one has $L_{4k}(\Z) \cong \Z$ and $L_{4k+2}(\Z) \cong \Z/2$.
I will use Cappell~\cite[Corollary~6]{Ca1,Ca2} and Shaneson's~\cite{Sh} calculations
\begin{align*}
  L_5(\Z[F_m]) &\cong L_5(\Z) \oplus \bigoplus^m \wt{L}_5(\Z[\Z]) \cong \oplus^m  L_4(\Z) \cong \oplus^m \Z; \\
L_6(\Z[F_m]) &\cong  L_6(\Z) \oplus \bigoplus^m \wt{L}_6(\Z[\Z]) \cong L_6(\Z) \oplus \bigoplus^m L_5(\Z) \cong  \Z/2.
\end{align*}

\begin{proposition}\label{prop:structure-set-trivial}
  The structure set  $\mathcal{S}(\natural^m S^1 \times D^4, \#^m S^1 \times S^3)$ is trivial, and hence the exteriors $W_0$ and $W_1$ are $s$-cobordant rel.\ boundary.
\end{proposition}

I will prove \cref{prop:structure-set-trivial} in the next two lemmas.  Write
\[(X,\partial X) := (\natural^m S^1 \times D^4, \#^m S^1 \times S^3).\]

\begin{lemma}\label{lemma:surgery-obstrn-injective}
  The surgery obstruction map $\mathcal{N}(X,\partial X) \to L_5(\Z[F_m]) \cong \oplus^m \Z$ is injective.
\end{lemma}

\begin{proof}
Let $f \colon W \to \natural^m S^1 \times D^4$, together with normal data $b$, represent an element of $\mathcal{N}(X,\partial X)$.
For each $k=1,\dots,m$, take a transverse inverse image of $\{1\} \times D^4$ in the $k$th boundary connected summand in $X=\natural^m S^1 \times D^4$. This gives a 4-manifold with boundary $S^3$, which I denote by $V^k$. The boundary is $S^3$ because $f|_{\partial W}$ is a diffeomorphism, and because $\partial V^k$ is the inverse image under~$f$ of $\partial(\{1\} \times D^4) = S^3$.
Then consider the signature $\sigma(V^k) \in \Z$, and use this to form a tuple $\big(\sigma(V^k)\big)_{k=1}^m \in \Z^m$.  This is preserved under normal bordism by cobordism invariance of the signature.
I claim that the degree one normal map $(f,b)$ is normally bordant to the identity map if $\sigma(V^k)=0$ for every $k=1,\dots,m$.

To see the claim I use the normal splitting theorem, \cite[Theorem~17.24]{LM}.
Recall that the set of normal bordism classes are in one to one correspondence with
$[(X,\partial X),(G/O,\ast)]$.
Let $Y_k \cong D^4$ be the $k$th $\{1\} \times D^4$ and let $i_k \colon Y_k \to X$ be the inclusion map. The normal splitting theorem gives a commuting square:
\[\begin{tikzcd}
    {[(X,\partial X),(G/O,*)]} \ar[d,"\cong"] \ar[r,"i_k^*"] & {[(Y_k,\partial Y_k),(G/O,*)]} \ar[d,"\cong"] \\
    \mathcal{N}(X,\partial X) \ar[r] & \mathcal{N}(Y_k,\partial Y_k)
  \end{tikzcd}\]
  where the top horizontal map is the restriction and the bottom horizontal map comes from transversality.
  Let $g \in [(X,\partial X),(G/O,*)]$ correspond to $(f,b) \in \mathcal{N}(X,\partial X)$ under the left vertical map.
Note that \[\mathcal{N}(Y_k,\partial Y_k) \cong \mathcal{N}(D^4,S^3) \cong [S^4,G/O] \cong 16\Z \hookrightarrow \Z,\] detected by the simply-connected surgery obstruction, which since $S^4$ has zero signature is the signature of the domain 4-manifold. So the bottom horizontal maps determine a map $\mathcal{N}(X,\partial X) \to \oplus_{k=1}^m \Z$ given by the collection of codimension one signatures.
If they all vanish, the  image of $(f,b)$ in $\mathcal{N}(Y_k,\partial Y_k)$ is trivial for all $k$, and hence by commutativity $i_k^*(g) \in [(Y_k,\partial Y_k),(G/O,*)]$ is null-homotopic for each~$k$. This yields a homotopy of~$g$ to a map $g'$ that is constant on the relative 4-skeleton $\partial X \cup (\sqcup_{k} Y_k)$ of $X$. This null-homotopy extends across the rest of $X$, which is a 5-cell, since  $\pi_5(G/O)=0$. 
Thus $g'$, and hence $g$, is homotopic rel.\ $\partial X$ to a constant map, meaning $(f,b)$ is  normally bordant rel.\ boundary to the identity.

Here, $\pi_5(G/O)$ can be computed using the long exact sequence
\[\cdots \to \pi_5(O) \to \pi_5(G) \to \pi_5(G/O) \to \pi_{4}(O) \to \pi_{4}(G) \to \cdots\]
arising from the fibration $G/O \to BO \to BG$ and the shifts $\pi_n(BG) \cong \pi_{n-1}(G)$ and $\pi_n(BO) \cong \pi_n(O)$. I know  $\pi_4(O) =0$ by Bott periodicity~\cite{Bo}, and also that $\pi_5(G) = \pi_5^s =0$.  Thus  $\pi_5(G/O) =0$ as asserted.
This completes the proof of the claim.

On the other hand the surgery obstruction group is $L_5(\Z[F_m]) \cong \oplus^m \wt{L}_5(\Z[\Z]) \cong \oplus^m L_4(\Z) \cong \oplus^m \Z$. I argue that the image of a degree one normal map under $\mathcal{N}(X,\partial X) \to L_5(\Z[F_m]) \xrightarrow{\cong} \oplus^m \Z $ is $\big(\sigma(V^k)\big)_{k=1}^m$.
Cappell~\cite{Ca1,Ca2} showed, using that $F_m$ has no order two elements, that there is a normal bordism to a degree one normal map where the domain is a boundary connected sum and the normal map splits as a boundary connected sum of normal maps $f_k \colon W_k \to S^1 \times D^4$, $k=1,\dots,m$. The surgery obstruction is then determined by the surgery obstructions of the $(W_k,f_k)$ in $\wt{L}_5(\Z[\Z])$, geometrically realising the isomorphism $L_5(\Z[F_m]) \cong \bigoplus^m \wt{L}_5(\Z[\Z])$.  Shaneson~\cite{Sh} showed that $\wt{L}_5(\Z[\Z]) \cong L_4(\Z)$, and since $L_4(\Z) \cong 8\Z$ he deduced that the surgery obstruction is given by a codimension one signature; in this case for each $k$ the obstruction is given by $\sigma(V_k)$ as above.

Thus if the surgery obstruction is zero, $\sigma(V^k)=0$ for $k=1,\dots,m$, and so $(f,b)$ is trivial in $\mathcal{N}(X,\partial X)$. Thus the surgery obstruction map  is injective as claimed.
\end{proof}

\begin{remark}
  I have presented a proof of \cref{lemma:surgery-obstrn-injective} via the splitting principle for surgery obstructions pioneered by Cappell and Shaneson.  However, the lemma can also be understood from a more general viewpoint, using that the surgery obstruction map factors as follows~\cite[Proposition~19.7]{LM}, where $\mathcal{N}^{\mathrm{Cat}}(X,\partial X)$ denotes the $\mathrm{Cat}$ normal invariants for $\mathrm{Cat} \in \{\mathrm{Diff},\mathrm{Top}\}$, and $\mathbb{L}$ is the Quinn-Ranicki spectrum. \[\mathcal{N}^{\mathrm{Diff}}(X,\partial X) \rightarrowtail \mathcal{N}^{\mathrm{Top}}(X,\partial X) \xrightarrow{\cong} \mathbb{L}\langle 1 \rangle_5(X) \xrightarrow{\cong} \mathbb{L}_5(X) \xrightarrow{\cong} \mathbb{L}_5(BF_m) \xrightarrow{\cong} L_5(\Z[F_m]).\]
\begin{enumerate}[(i)]
  \item The first map is injective because the kernel is $[(X,\partial X),(\mathrm{Top}/O,*)] \cong H^3(X,\partial X;\Z/2) \cong H_2(X;\Z/2) =0$.
Here I use that there is a fibration $\Top/O \to G/O \to G/\Top$, that $X$ is 5-dimensional, and that there is a 7-connected map $\Top/O \to K(\Z/2,3)$.
\item The identification $\mathcal{N}^{\mathrm{Top}}(X,\partial X) \xrightarrow{\cong} \mathbb{L}\langle 1 \rangle_5(X)$ is part of the general theory of the Quinn-Ranicki $\mathbb{L}$-spectrum. I refer to \cite[Proposition~19.7]{LM} for further citations.
  \item The passage $\mathbb{L}\langle 1 \rangle_5(X) \xrightarrow{\cong} \mathbb{L}_5(X)$ from the $[1,\infty]$ truncation of $\mathbb{L}$ to the full $\mathbb{L}$-spectrum can be computed to be an isomorphism using the Atiyah--Hirzebruch spectral sequence and the observation that for $R$ an abelian group $H_p(X;R)=0$ for $p \geq 2$. 
  \item  The next map $\mathbb{L}_5(X) \xrightarrow{\cong} \mathbb{L}_5(BF_m)$ is an isomorphism because it is induced by a homotopy equivalence $X \simeq BF_m$.
  \item The assembly map $\mathbb{L}_5(BF_m) \xrightarrow{\cong} L_5(\Z[F_m])$ is known to be an isomorphism because $F_m$ is torsion-free and satisfies the Farrell-Jones conjecture (see e.g.~\cite[Theorem~16.44]{LM}).
\end{enumerate}
\end{remark}

\begin{lemma}
  The action of $L_6(\Z[F_m]) \cong  \Z/2$ on  $\mathcal{S}(X,\partial X)$ is trivial.
\end{lemma}

\begin{proof}
The Wall realisation action is trivial if and only if the surgery obstruction map $\mathcal{N}(X \times I,\partial (X \times I)) \to L_6(\Z[F_m]) \cong L_6(\Z) \cong \Z/2$ is surjective.
The nontrivial element of $L_6(\Z)$ is realised by a degree one normal map $(F,B) \colon S^3 \times S^3 \to S^6$~\cite[Proposition~8.181~and~Exercise~8.186]{LM}. So, starting with the identity degree one normal map over $\natural^m S^1 \times D^4 \times I$, take the connected sum in domain and codomain with $(F,B)$, to obtain a degree one normal map with surgery obstruction the generator of $L_6(\Z[F_m]) \cong \Z/2$.
\end{proof}

\begin{proof}[Proof of Proposition~\ref{prop:structure-set-trivial}]
  The first part  now follows from the previous two lemmas and exactness of the surgery sequence \eqref{eqn:ses}. For the second part of Proposition~\ref{prop:structure-set-trivial}, Lemma~\ref{lemma:DONM} implies that the exteriors $W_i := D^5 \sm \nu D^i_m$ fit into elements $f_i \colon W_i \xrightarrow{\simeq} \natural^m S^1 \times D^4$ of the relative structure set $\mathcal{S}(X,\partial X)$. Since that set is a singleton,   the exteriors $W_0$ and $W_1$ are $s$-cobordant rel.\ boundary.
\end{proof}

\section{Proof of Theorem~\ref{thm}}\label{sec:proof-thm-A}

By the $s$-cobordism theorem~\cite{Sm,Mi,Ba,Ma,St2} there is a diffeomorphism $W_0 \cong W_1$ that restricts to the composite $\partial W_0 \cong \#^m S^1 \times S^3 \cong \partial W_1$, which is  the diffeomorphism $\varphi$ constructed in the proof of Lemma~\ref{lemma:DONM}. The boundary of $W_i \subseteq D^5$ splits as
\[\partial W_i
\cong S^4 \sm \nu U_m \cup \bigsqcup^m D^3 \times S^1.\]
 Glue in $\sqcup^m D^3 \times D^2$ along $\sqcup^m D^3 \times S^1 \subseteq \partial W_i$. This recovers $D^5$, with $\sqcup^m D^3 \times \{0\}$ sent to $D^i_m$.

Extend the diffeomorphism $W_0 \cong W_1$ across $\sqcup^m D^3 \times D^2$, to obtain a diffeomorphism $\Psi \colon D^5 \to D^5$.
The original identification $\varphi\colon \partial W_0 \cong \partial W_1$ from the proof of Lemma~\ref{lemma:DONM} was obtained by restricting a diffeomorphism
$S^4 \times I \cup \ol{\nu} D^0_m \xrightarrow{\cong} S^4 \times I \cup \ol{\nu} D^1_m$ that restricts to the identity on $S^4 \times \{0\}$ and maps $D^0_m$ to $D^1_m$. Hence one can arrange that $\Psi$ restricts to the identity on $\partial D^5 = S^4$ and maps $D^0_m$ to $D^1_m$.

Now, every diffeomorphism of the 5-disc that is the identity on $S^4$ is isotopic rel.\ boundary to the identity (by \cite{Ce} and \cite[Corollary~VIII.5.6]{K}, this is equivalent to there being no exotic 6-spheres~\cite{KM}). Thus  $\Psi$ is isotopic rel.\ boundary to the identity. The resulting isotopy $\Phi_t \colon D^5 \to D^5$ satisfies that $\Phi_0 = \Psi$, $\Phi_t|_{S^4} = \Id_{S^4}$ for all $t \in [0,1]$, and $\Phi_1 = \Id_{D_5}$.  Therefore $\Phi_0(D^0_m) = \Psi(D^0_m) = D^1_m$, while $\Phi_1(D^0_m) = \Id(D^0_m) = D^0_m$.  It follows that $D^t_m := \Phi_t(D^0_m)$ is a 1-parameter smooth family of collections of smoothly embedded 3-discs, with boundary $U_m$, interpolating between $D^1_m$ and $D^0_m$. Thus $D^0_m$ and $D^1_m$ are smoothly isotopic rel.\ $U_m$.

\begin{remark}
It was not necessary to apply the fact that the rel.\ boundary smooth mapping class group of~$D^5$ is trivial. If I pretend I do not know this, instead compose the given diffeomorphism $\Psi$ with a map isotopic to $\Psi^{-1}$, but with the inverse shrunk down to be supported in a small $D^5$ away from $\cup D^i_m$. This composite is isotopic to the identity, and again the isotopy carries $D^1_m$ to~$D^0_m$.
  \end{remark}


\end{document}